%% file: paper.tex
\documentclass{article}
\pdfoutput=1
\usepackage{parskip}
\usepackage{amsmath,amsfonts,amssymb}
\usepackage{amsthm}
\usepackage{mathtools}
\usepackage{microtype}
\usepackage{booktabs}
\usepackage{enumitem}
\usepackage{algorithm}
\usepackage{algpseudocode}
\usepackage[colorlinks=true,linkcolor=blue,citecolor=blue,urlcolor=blue]{hyperref}
\usepackage{orcidlink}
\usepackage[capitalise]{cleveref}
\usepackage[auth-sc, affil-it]{authblk}
\usepackage{geometry}
\def\R{\mathbb{R}}
\def\argmin{\mathrm{argmin}}

\newtheorem{lemma}{Lemma}
\newtheorem{proposition}{Proposition}

\theoremstyle{definition}

\theoremstyle{remark}
\newtheorem{remark}{Remark}
\newlist{assumpenum}{enumerate}{1}
\setlist[assumpenum]{label=(A\arabic*),ref=A\arabic*,leftmargin=*,align=left}
\crefname{assumpenumi}{Assumption}{Assumptions}
\Crefname{assumpenumi}{Assumption}{Assumptions}
\crefname{assumption}{Assumption}{Assumptions}
\Crefname{assumption}{Assumption}{Assumptions}

\begin{document}
\title{A polynomially accelerated fixed-point iteration for vector problems}
\date{\small Published in \textit{e-Journal of Analysis and Applied Mathematics} \textbf{2025} (2025), 44--56.\\
DOI: \href{https://doi.org/10.62780/ejaam/2025-005}{10.62780/ejaam/2025-005}}
\author[1]{\orcidlink{0000-0003-1065-2590}Francesco Alemanno\thanks{Corresponding author. Email: francesco.alemanno@unisalento.it. Postal address: Technical office, SolarSud T.R.E., Via Tiziano Vecellio 11, San Giorgio Ionico (TA), Italy.}}
\affil[1]{\textit{Technical office, SolarSud T.R.E., Via Tiziano Vecellio 11, San Giorgio Ionico (TA), Italy}}

\maketitle

\begin{abstract}
\noindent\textbf{~~Abstract~~}
Fixed-point solvers are ubiquitous in nonlinear PDEs, yet their progress collapses whenever the Jacobian at the solution carries an eigenvalue arbitrarily close to one. We ask whether such stagnation can be removed without storing long histories or solving dense least squares. Under two assumptions---(A1) the linearised error $e_n$ is dominated by a multiplier $m$ with $|m|<1$ and (A2) residuals shrink monotonically---we construct a quadratic blend of three iterates whose error polynomial has a double root at $m$. This three-point polynomial accelerator (TPA) cancels the stubborn mode up to $o(\|e_n\|)$, reduces to Aitken's $\Delta^2$ process in one dimension, and matches a doubly blended Anderson step with depth $m=2$ when the regularisation vanishes, yet it keeps the Picard memory footprint. The only extra ingredient is a residual-based estimate of $w=(1-m)^{-1}$ obtained from a closed-form regularised least-squares fit that remains stable even when two residuals nearly coincide. Numerical experiments on linear systems with clustered spectra, a $320$-dimensional nonlinear $\tanh$ fixed point, and a $50\times 50$ Poisson discretisation show that TPA reaches the $10^{-8}$ residual tolerance in $32$, $36$, and $244$ map evaluations (respectively). In the same settings SOR requires $663$ steps and Anderson acceleration with depth $m=5$ consumes $52$, $38$, and $955$ evaluations. TPA therefore supplies a parameter-free, constant-memory drop-in accelerator whenever a single contraction factor throttles convergence.


\end{abstract}
\tableofcontents

\section{Introduction}
Fixed-point iterations of the form
\begin{equation}
    \label{eq:picard}
    x_{k+1} = T(x_k),
\end{equation}
constitute a standard tool for solving nonlinear systems and discretised partial differential equations. It is widely understood that the spectral structure of the Jacobian of $T$ at the fixed point $x_\star$ dictates their effectiveness, stagnation is common when a single mode contracts markedly more slowly than the others. Classical remedies---such as relaxation strategies or extrapolation methods---attenuate this difficulty, yet they typically demand problem-dependent parameter tuning or introduce dense linear-algebra kernels that compromise the economy of the underlying iteration.

In this work we develop a \emph{three-point polynomial accelerator} (TPA) that preserves the austere structure of a Picard loop (i.e. the plain iteration of \eqref{eq:picard}) while restoring rapid convergence. The mechanism interrogates successive residuals, forms an on-the-fly estimate of the dominant error contribution, and produces a regularised three-point update of the most recent iterates so that such error contribution is annihilated. Moreover, since the routine only manipulates three vectors at any time and eschews expanding history windows, it integrates seamlessly within existing implementations and adds a negligible computational burden.

From a formal standpoint, the analysis that leads to TPA builds upon the classical literature on sequence transformations. Aitken's $\Delta^2$ process removes the leading error term in one dimension~\cite{aitken_xxv.bernoullis_1927}; minimal-polynomial extrapolation transfers the annihilating-polynomial perspective to vector-valued settings~\cite{sidi_vector_2008}; and Anderson acceleration constructs optimal mixing coefficients via constrained least squares~\cite{anderson_iterative_1965}. Successive over-relaxation~\cite{young_iterative_1954} remains the canonical spectral shaper but hinges on careful damping-parameter selection. A comprehensive synthesis of these acceleration paradigms and their extensions is presented in the survey~\cite{saad_acceleration_2025}. In order to make this bridge between our algorithm and pre-existing methods, in \Cref{sec:aitken,sec:anderson} we summarise its precise relation to Aitken's $\Delta^2$ process and Anderson acceleration of depth two.

\paragraph{In a nutshell:} The main contributions of this manuscript are summarised as follows:
\begin{itemize}
    \item the derivation of a constant-memory polynomial accelerator that elevates Aitken's principle to vector settings without dense auxiliary solves;
    \item a suite of linear and nonlinear experiments demonstrating consistent reductions in map evaluations relative to Picard, SOR, and Anderson acceleration, with the advantage over deeper Anderson variants widening as the test problems grow more challenging, all while maintaining minimal overhead.
\end{itemize}

We now move on to the formal derivation of the accelerator, establishing the assumptions and algebraic structure that underpin the subsequent analysis.

\section{Derivation of the Algorithm}
We study the fixed-point iteration
\begin{equation}
    x_{n+1} = T(x_n), \quad x_n \in \R^d,
\end{equation}
around a fixed point $x_\star = T(x_\star)$. Denote the error by $e_n = x_n - x_\star$. 
\paragraph{Key assumptions.}
We work under two asymptotic conditions on the error sequence.
There exist $|m|<1$ and an index $n_0$ such that for all $n \geq n_0$ the errors obey
\begin{assumpenum}
    \item \label{ass:A1} $e_{n+1} = m e_n + o(\|e_n\|)$;
    \item \label{ass:A2} $\|e_{n+1}\| < \|e_n\|$.
\end{assumpenum}

\Cref{ass:A1} states that the linearisation of $T$ at $x_\star$ is governed by a single dominant multiplier $m$; \cref{ass:A2} ensures that the error magnitude shrinks, so the higher-order terms in \cref{ass:A1} remain subordinate.

\begin{remark}[Jacobian interpretation of Assumption~\ref{ass:A1}]\label{rem:jacobian}
When $T$ is differentiable at $x_\star$, let $J = T'(x_\star)$ denote its Jacobian. The linearisation gives $e_{n+1} = J e_n + o(\|e_n\|)$. If $J$ admits a dominant mode with right and left vectors $u$ and $v$, then
\begin{equation}
    J = m\frac{u\,v^\top}{\langle u, v \rangle} + \text{higher-order terms},
\end{equation}
so the scalar multiplier $m$ in \cref{ass:A1} coincides with the leading spectral component that governs the residual recursion.
\end{remark}

\paragraph{Quadratic blend of three iterates.}
Let $y_3 = x_n$, $y_2 = x_{n-1}$, and $y_1 = x_{n-2}$ with errors $e_3$, $e_2$, $e_1$ respectively. Any quadratic blend preserving constants can be parametrised as
\begin{equation}\label{eq:blend-def}
    y_+(y_1,y_2,y_3;a,b) = (1-a-b) y_1 + a y_2 + b y_3.
\end{equation}

\begin{lemma}[Error polynomial]\label{lem:error-poly}
Under \cref{ass:A1,ass:A2} the error committed by the blend \eqref{eq:blend-def} satisfies
\begin{equation}\label{eq:blend-error}
    y_+ - x_\star = \bigl(1 - a - b + a m + b m^2\bigr) e_1 + o(\|e_1\|).
\end{equation}
\end{lemma}

\begin{proof}
\Cref{ass:A1} yields $e_2 = m e_1 + o(\|e_1\|)$. Applying \cref{ass:A1} once more and using \cref{ass:A2} to bound the higher-order term gives $e_3 = m e_2 + o(\|e_2\|) = m^2 e_1 + o(\|e_1\|)$. Substituting $y_i = x_\star + e_i$ into \eqref{eq:blend-def} we obtain
\begin{align*}
    y_+ - x_\star &= (1-a-b) e_1 + a e_2 + b e_3 \\
    &= \bigl(1 - a - b + a m + b m^2\bigr) e_1 + o(\|e_1\|),
\end{align*}
which is \eqref{eq:blend-error}.
\end{proof}
\paragraph{Optimality conditions.}
To remove the dominant error mode to second order we enforce that the prefactor in \eqref{eq:blend-error} vanishes together with its derivative with respect to $m$. The conditions
\begin{equation}\label{eq:double-root-conditions}
    1 - a - b + a m + b m^2 = 0, \quad a + 2 b m = 0,
\end{equation}
force the error polynomial to have a double root at the unknown multiplier $m$.

\begin{proposition}[Optimal quadratic coefficients]\label{prop:optimal-coefficients}
Under \cref{ass:A1,ass:A2}, solving \eqref{eq:double-root-conditions} gives
\begin{equation}\label{eq:coefficients}
    a = -\frac{2 m}{(1-m)^2}, \quad b = \frac{1}{(1-m)^2}.
\end{equation}
Consequently the blend \eqref{eq:blend-def} with the optimal coefficients reduces to
\begin{equation}\label{eq:quadratic-blend}
\begin{split}
    y_+ = y_1 + 2 w (y_2 - y_1) + w^2 (y_1 - 2 y_2 + y_3),\\
    \mathrm{with~~}w := \frac{1}{1-m},
\end{split}
\end{equation}
eliminates the dominant mode up to $o(\|e_1\|)$ and collapses to $x_\star$ when \cref{ass:A1} holds exactly.
\end{proposition}

\paragraph{Residual dynamics.}
Define the residuals $r_n = x_{n+1} - x_n = e_{n+1} - e_n$. The next lemma links $m$ to the observable sequence $(r_n)$.

\begin{lemma}[Residual recursion]\label{lem:residual-recursion}
Under \cref{ass:A1,ass:A2} we have
\begin{equation}\label{eq:residual-recursion}
    r_{n+1} = m r_n + o(\|e_n\|).
\end{equation}
\end{lemma}

\begin{proof}
Using \cref{ass:A1} twice yields $e_{n+2} = m e_{n+1} + o(\|e_{n+1}\|)$ and $e_{n+1} = m e_n + o(\|e_n\|)$. Therefore
\begin{align*}
    r_{n+1} &= e_{n+2} - e_{n+1} = (m-1) e_{n+1} + o(\|e_{n+1}\|) \\
            &= (m-1) (m e_n + o(\|e_n\|)) + o(\|e_n\|) \\
            &= m (m-1) e_n + o(\|e_n\|).
\end{align*}
Similarly $r_n = (m-1) e_n + o(\|e_n\|)$, whence \eqref{eq:residual-recursion} follows.
\end{proof}

Subtracting successive residuals and applying \eqref{eq:residual-recursion} gives
\begin{equation}\label{eq:residual-gap}
    r_n - r_{n+1} = (1-m) r_n + o(\|e_n\|).
\end{equation}
Combining \eqref{eq:residual-gap} with the $w = (1-m)^{-1}$ from \eqref{eq:quadratic-blend} shows that
\begin{equation}\label{eq:w-from-residuals}
    w (r_n - r_{n+1}) = r_n + o(\|e_n\|).
\end{equation}
For the triple $(y_1,y_2,y_3)$ this relation reads
\begin{equation}\label{eq:data-equation}
    w (r_1 - r_2) = r_1 + o(\|e_1\|), \quad r_1 = y_2 - y_1, \quad r_2 = y_3 - y_2.
\end{equation}

\paragraph{Regularised estimation of $w$.}
We infer $w$ by fitting \eqref{eq:data-equation} in least squares with a quadratic regulariser that favours $w \approx 1$ when the residual gap is nearly singular:
\begin{equation}\label{eq:regularised-problem}
    \widehat{w}_\theta = \argmin_{w \in \R} \left\|w (r_1 - r_2) - r_1 \right\|_2^2 + \theta^2 (w-1)^2, 
\end{equation}
with $\theta > 0$.
The first-order optimality condition is
\begin{equation}
    \bigl(\|r_1 - r_2\|_2^2 + \theta^2\bigr) w = \langle r_1 - r_2, r_1 \rangle + \theta^2,
\end{equation}
yielding the closed-form estimate
\begin{equation}\label{eq:w-hat}
    \widehat{w}_\theta = \frac{\langle r_1 - r_2, r_1 \rangle + \theta^2}{\| r_1 - r_2 \|_2^2 + \theta^2}.
\end{equation}
When the scalar model of \cref{ass:A1} is exact, $r_2 = m r_1$ and \eqref{eq:w-hat} recovers $w = (1-m)^{-1}$.

Inserting $\widehat{w}_\theta$ into \eqref{eq:quadratic-blend} produces the practical accelerator
\begin{equation}
    y_+ = y_1 + 2 \widehat{w}_\theta (y_2 - y_1) + \widehat{w}_\theta^{\,2} (y_1 - 2 y_2 + y_3),
\end{equation}
which inherits the dominant-mode cancellation while remaining stable when $r_1$ and $r_2$ nearly coincide.

The pseudocode in \cref{alg:poly-accel} consolidates these steps.

\begin{algorithm}[h]
    \caption{TPA, Three-point polynomial accelerator}
    \label{alg:poly-accel}
    \begin{algorithmic}[1]
        \Require map $T$, initial iterate $x_0$, tolerance $\texttt{tol}$, maximum loop count $K$, regularisation $\theta>0$
        \State $y_1 \gets x_0$
        \State $y_2 \gets T(y_1)$, $y_3 \gets T(y_2)$
        \State $\rho \gets \|y_3 - y_2\|_\infty$
        \If{$\rho < \texttt{tol}$} \Return $y_3$ \EndIf
        \For{$i = 1,\ldots,K$}
            \State $r_1 \gets y_2 - y_1$, \quad $r_2 \gets y_3 - y_2$
            \State $w \gets \frac{\langle r_1 - r_2, r_1 \rangle + \theta^2}{\| r_1 - r_2 \|_2^2 + \theta^2}$ \Comment{$\theta$ stabilises the fit}
            \State $y_1 \gets y_1 + 2 w (y_2 - y_1) + w^2 (y_1 - 2 y_2 + y_3)$
            \State $y_2 \gets T(y_1)$, \quad $y_3 \gets T(y_2)$
            \State $\rho \gets \|y_3 - y_2\|_\infty$
            \If{$\rho < \texttt{tol}$} \Return $y_3$ \EndIf
        \EndFor
        \State \Return $y_3$ \Comment{Maximum iterations reached without convergence}
    \end{algorithmic}
\end{algorithm}

Each loop costs two evaluations of $T$ and a handful of inner products, and the method stores only three iterates and two residuals.
The next section evaluates this accelerator on nonlinear vector fixed-point problems, comparing it with Picard iteration in terms of contraction per step and overall work.

\section{Numerical Experiments}
We investigate TPA on three representative fixed-point problems: a linear system with a clustered spectrum, a nonlinear $\tanh$ map, and a discretised Poisson equation. All solvers start from the zero vector and terminate once the residual criterion $\|T(x_k) - x_k\|_\infty < \texttt{tol}$ with $\texttt{tol} = 10^{-8}$ is satisfied or after $10^6$ evaluations of the map $T$, and we keep the regularisation parameter fixed at $\theta = 10^{-9}$ across all trials.
To anchor the comparisons we also compute a reference fixed point in closed form for each scenario, enabling us to report both the final residual $\|T(x_k) - x_k\|_\infty$ and the final error $\|x_k - x_\star\|_\infty$.

Before turning to the individual benchmarks it is useful to contrast the cost structure of the competitors. Relative to Anderson acceleration of depth $m$, which stores $m$ iterates and solves a dense $m\times m$ least-squares problem each step, TPA requires only three iterates, two residuals, and no dense solves. Compared with weighted Jacobi or SOR, TPA selects the effective damping factor automatically from the residuals and removes per-problem tuning.

\subsection{Linear System with Clustered Spectrum}
We first test TPA on a linear system whose contraction factors cluster in $[0.9, 0.99]$, a regime where single-parameter relaxations perform poorly. The test map is built in dimension $80$ by drawing a random orthogonal matrix $Q$ from the QR factorisation of a standard-normal matrix and forming $M = Q \,\mathrm{diag}(\lambda)\, Q^\top$ with evenly spaced eigenvalues $\lambda \in [0.9, 0.99]$. We synthesise a fixed point $x_\star$ from the same random generator and apply the affine map $T(x) = Mx + c$ with $c = x_\star - M x_\star$. \cref{tab:clustered} compares TPA with SOR using relaxation $\omega = 1.8$, Picard iteration ($\omega = 1$), and Anderson acceleration at depths $m\in\{2,3,5\}$, while \cref{fig:clustered-residuals} reports the convergence histories.

\begin{table}[h]
\footnotesize
    \centering
    \caption{\textbf{Linear system with clustered spectrum.} Performance of TPA compared with Picard, SOR, and Anderson($m=2,3,5$).}
    \label{tab:clustered}
    \input{outputs/clustered_spectrum/clustered_results.tex}
\end{table}

\begin{figure}[h]
    \centering
    \includegraphics[width=0.6\linewidth]{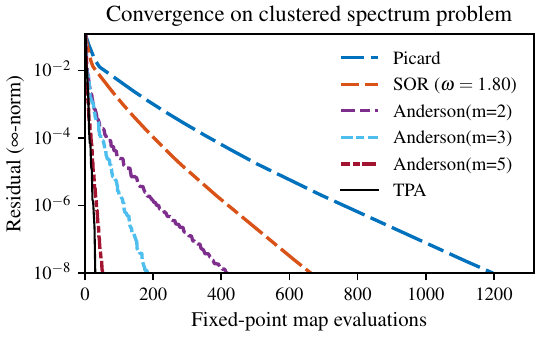}
    \caption{\textbf{Linear system with clustered spectrum.} Residual histories showing that TPA damps the dominant modes faster than the reference methods.}
    \label{fig:clustered-residuals}
\end{figure}

In this regime TPA converges in $32$ map evaluations, whereas SOR and Picard require $663$ and $1197$ iterations to reach the same tolerance. To calibrate the comparison we also examine Anderson acceleration, which shortens the run to $419$ evaluations at depth two, $183$ at depth three, and $52$ at depth five, yet TPA still uses roughly one thirteenth of the work of the shallowest variant and about $40\%$ fewer evaluations than the deepest scheme while avoiding dense $m\times m$ solves. \cref{fig:clustered-residuals} reinforces this gap by showing the smoother decay produced by TPA relative to the plateaus observed for the competing schemes. Although this instance is nearly monotone, monotonicity is not ensured in general; consequently the regularised coefficient estimate in \eqref{eq:w-hat} moderates the blend when residual gaps are small, and all convergent methods reach comparable final errors.

\subsection{Nonlinear \textit{tanh} Fixed Point}
We next apply TPA to the nonlinear map $T(x) = \tanh(Bx + c)$ with $\|x_\star\|_{\infty} < 1$, whose Jacobian spectrum near the solution again clusters near unity. We set the dimension to $320$, draw $B = Q \,\mathrm{diag}(\lambda)\, Q^\top$ with eigenvalues evenly spaced in $[0, 0.999]$, and choose a reference fixed point $x_\star$ by sampling each component uniformly in $[-0.7, 0.7]$. The offset is $c = \operatorname{arctanh}(x_\star) - B x_\star$, so $x_\star$ is recovered exactly when $T$ converges. \cref{tab:tanh} reports the iteration counts for Picard, SOR with $\omega=1.5$, Anderson acceleration with depths $m \in \{2,3,5\}$, and TPA; \cref{fig:tanh-residuals} shows the residual traces.

\begin{table}[h]
\footnotesize
    \centering
    \caption{\textbf{Nonlinear $\tanh$ fixed point.} Iteration counts for Picard, SOR, Anderson, and TPA.}
    \label{tab:tanh}
    \input{outputs/tanh_fixed_point/tanh_results.tex}
\end{table}

\begin{figure}[h]
    \centering
    \includegraphics[width=0.6\linewidth]{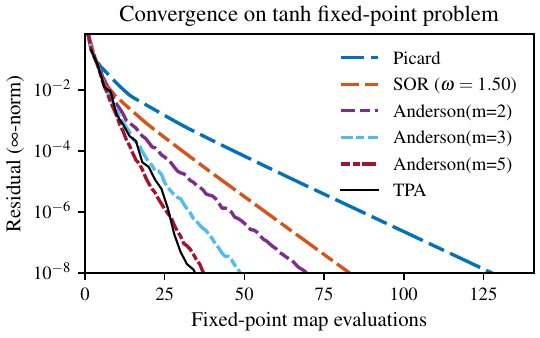}
    \caption{\textbf{Nonlinear $\tanh$ fixed point.} Residual histories showing the smooth convergence of TPA relative to the baselines.}
    \label{fig:tanh-residuals}
\end{figure}

All methods level off near $10^{-8}$ because they share the same stopping tolerance rather than any intrinsic instability, so iteration counts provide the relevant comparison. Within that framework TPA converges in $36$ evaluations, outpacing Anderson acceleration with depth two ($70$ evaluations) and depth three ($49$ evaluations) while still edging the depth-five variant ($38$ evaluations) and avoiding the larger least-squares solves and storage burden. The resulting three-point update therefore preserves smooth decay without depth tuning.

\subsection{Discretised 2D Poisson Equation}
The final experiment considers a $50\times 50$ standard five-point finite-difference discretisation of the Poisson equation $-\Delta u = f$ on the unit square with homogeneous Dirichlet boundary conditions. The Jacobi fixed-point map arises from the usual diagonal splitting of the stiffness matrix on a grid with mesh spacing $h = 1/51$. The right-hand side is prescribed deterministically as
\begin{equation}
    f(x,y) = \sin(\pi x^2)\,\sin(2\pi y^2)
\end{equation}
evaluated at the interior grid points. We assemble the resulting Kronecker-structured linear system and precompute the exact grid solution $x_\star$ for error reporting. \cref{tab:poisson} summarises the performance of Jacobi ($\omega=1$), Anderson acceleration at depths $m\in\{2,3,5\}$, and TPA; \cref{fig:poisson-residuals} plots the residual traces.

\begin{table}[h]
\footnotesize
    \centering
    \caption{\textbf{Discretised 2D Poisson equation.} Performance on the $50\times 50$ grid for Jacobi (Picard), Anderson, and TPA.}
    \label{tab:poisson}
    \input{outputs/poisson_2d/poisson_results_50.tex}
\end{table}

\begin{figure}[htbp!]
    \centering
    \includegraphics[width=0.6\linewidth]{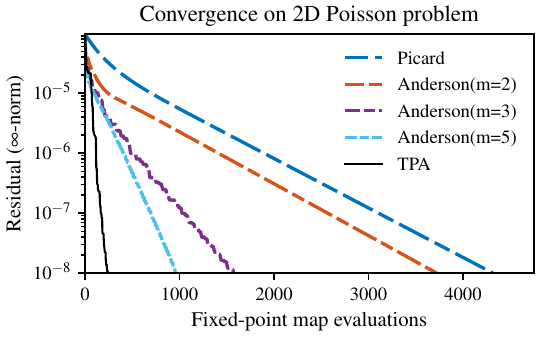}
    \caption{\textbf{Discretised 2D Poisson equation.} Residual histories highlighting the faster damping achieved by TPA.}
    \label{fig:poisson-residuals}
\end{figure}

On this discretised elliptic problem TPA reaches the stopping tolerance in $244$ fixed-point evaluations. Anderson acceleration requires $3722$ evaluations at depth two, $1573$ at depth three, and $955$ at depth five, so even the strongest baseline still needs nearly four times as many iterations, whereas Jacobi demands $4317$ steps, more than an order of magnitude beyond TPA. The widening gap relative to the depth-five Anderson method on this more demanding problem underscores the advantage conferred by this three-point update in two-dimensional elliptic problems with slowly decaying error components. These benchmarks motivate a closer comparison between TPA and the classical extrapolation routines it refines; the next section formalises this connection.

\section{Relation to earlier work}

This section situates TPA within the literature on polynomial and residual-based acceleration. We first recall how the scalar restriction and vanishing regularisation recover Aitken's $\Delta^2$ transform, and then examine the structural link with depth-two Anderson acceleration.

\subsection{Relation to Aitken's $\Delta^2$ Process}\label{sec:aitken}
Setting the regularisation to zero and restricting to scalar iterates recovers Aitken's classical accelerator. Let $y_1 = x_{n-2}$, $y_2 = x_{n-1}$, and $y_3 = x_n$ be three consecutive iterates of a scalar map. Their residuals are $r_1 = y_2 - y_1$ and $r_2 = y_3 - y_2$, and with $\theta = 0$ the least-squares fit \eqref{eq:regularised-problem} simplifies to
\begin{equation}
    \widehat{w}_0 = \frac{\langle r_1 - r_2, r_1 \rangle}{\|r_1 - r_2\|_2^2} = \frac{r_1}{r_1 - r_2}.
\end{equation}
Substituting this coefficient into the quadratic blend \eqref{eq:quadratic-blend} yields
\begin{align}
    y_+ &= y_1 + 2 \widehat{w}_0 (y_2 - y_1) + \widehat{w}_0^{\,2} (y_1 - 2 y_2 + y_3) \\      
    &= y_1 - \frac{(y_2 - y_1)^2}{y_3 - 2 y_2 + y_1},
\end{align}
which coincides with Aitken's $\Delta^2$ transform applied to $x_{n-2}$.

\subsection{Relation to Anderson Acceleration}\label{sec:anderson}
Consider Anderson acceleration of depth two with iterates $y_1, y_2, y_3$ and residuals $r_1 = y_2 - y_1$ and $r_2 = y_3 - y_2$.  The depth-two update is obtained by choosing a mixing coefficient $a$ and forming the blend
\begin{equation}
    \label{eq:anderson-2}
    y_{+2} = (1-a) y_2 + a y_3.
\end{equation}
The coefficient $a$ is selected by the regularised least-squares fit
\begin{equation}
    \widehat{a}_\theta = \argmin_{a \in \R} \left\|(1-a) r_1 + a r_2\right\|_2^2 + \theta^2 \bigl[a^2 + (1-a)^2\bigr],
\end{equation}
where the residuals $r_i$ act as a proxy for the error of $y_{i+1}$ and the regularisation parameter $\theta$ avoids ill-conditioning.

Differentiating the objective and setting the derivative to zero yields the closed form
\begin{equation}
    \widehat{a}_\theta = \frac{\langle r_1 - r_2, r_1 \rangle + \theta^2}{\|r_1 - r_2\|_2^2 + 2 \theta^2}.
\end{equation}
With the optimal coefficient in hand, the accelerated iterate \eqref{eq:anderson-2} of Anderson's scheme is complete.

In order to connect to TPA we may also blend the earlier pair of iterates to obtain
\begin{equation}
    y_{+1} = (1-\widehat{a}_\theta) y_1 + \widehat{a}_\theta y_2,
\end{equation}
and then apply the same coefficient once more to combine $y_{+1}$ and $y_{+2}$:
\begin{equation}
    y_{++} = (1-\widehat{a}_\theta) y_{+1} + \widehat{a}_\theta y_{+2}.
\end{equation}
Expanding the last expression reveals the quadratic weights applied to the original iterates,
\begin{equation}
    y_{++} = (1-\widehat{a}_\theta)^2 y_1 + 2 \widehat{a}_\theta (1-\widehat{a}_\theta) y_2 + \widehat{a}_\theta^{\,2} y_3.
\end{equation}
When $\theta = 0$ the coefficient reduces to $\widehat{a}_0 = \widehat{w}_0$, so the three weights coincide with those in \eqref{eq:quadratic-blend} and the doubly blended iterate $y_{++}$ matches TPA exactly.  Depth-two Anderson acceleration, however, typically outputs only $y_{+2}$ and therefore misses the additional cancellation provided by the second blend. For $\theta>0$ the regularisation in the denominator also deviates slightly from the safeguard in \eqref{eq:w-hat}, yet this effect is dominated by the structural distinction between a single and a double blend.

The heuristic comparison above is intended to expose the algebraic kinship between the methods; a full proof that follows this route is feasible with additional work but falls outside the present scope. We close by distilling these implications for practice and outlining directions for refinement in the concluding section.

\section{Conclusions}
We have shown that a closed-form, constant-memory three-point polynomial accelerator (TPA) can markedly reduce the work required by fixed-point iterations across linear and nonlinear problems. On the clustered linear system TPA attains the convergence tolerance in only $32$ map evaluations compared with $663$ for SOR, $1197$ for Picard, and $419/183/52$ for Anderson acceleration with depths $m=2/3/5$, thereby trimming roughly $40\%$ of the evaluations relative to the deepest competitor while avoiding dense solves. The nonlinear $\tanh$ benchmark in turn converges in $36$ evaluations, maintaining a lead over Anderson acceleration at depths two ($70$), three ($49$), and five ($38$), and the Poisson experiment achieves the tolerance in $244$ iterations versus $3722/1573/955$ for Anderson($m=2/3/5$). Collectively these observations highlight a widening advantage over Anderson acceleration as the test problems grow more demanding and underscore the benefits of error-polynomial shaping informed by short-history contraction-factor estimates.

TPA is most effective when convergence is throttled by one slowly decaying mode: the residuals align, the contraction-factor estimate stabilises, and the resulting three-point update suppresses the stubborn component. Empirically it also performs well when the spectrum is only moderately clustered, provided that the dominant mode does not vary rapidly between successive steps, and this operational envelope motivates the safeguards below.

\paragraph{Limitations and safeguards.} Strongly non-normal Jacobians can still produce alternating dominant modes that temporarily mislead the residual ratio. Users should therefore monitor the residual norm and fall back to the base iteration if acceleration induces growth. In addition, purely polynomial extrapolations such as TPA cannot stabilise iterations whose local Jacobian has eigenvalues of magnitude greater than one, so a contractive regime remains indispensable.

Looking ahead, future work includes adaptive selection of the contraction-factor regularisation, richer surrogates derived from longer iterate histories, and higher-degree extensions that exploit additional iterates without sacrificing the simplicity that renders TPA attractive in practice.

\section*{Acknowledgements}
We thank Miriam Aquaro (\orcidlink{0000-0003-4836-0307}) for carefully reading the manuscript and providing valuable suggestions.
%%\bibliographystyle{plain}
%%\bibliography{bib}

\end{document}

%% file: outputs/clustered_spectrum/clustered_results.tex
\begin{tabular}{lrrr}
\toprule
Method & N. evals & $\|T(x) - x\|_\infty$ & $\|x - x_\star\|_\infty$ \\
\midrule
TPA & \textbf{32} & \textbf{$0.16\times 10^{-8}$} & \textbf{$0.12\times 10^{-6}$} \\
Picard & 1197 & $1.00\times 10^{-8}$ & $0.98\times 10^{-6}$ \\
SOR ($\omega=1.80$) & 663 & $0.99\times 10^{-8}$ & $0.97\times 10^{-6}$ \\
Anderson(m=2) & 419 & $0.98\times 10^{-8}$ & $0.86\times 10^{-6}$ \\
Anderson(m=3) & 183 & $0.95\times 10^{-8}$ & $0.94\times 10^{-6}$ \\
Anderson(m=5) & 52 & $0.94\times 10^{-8}$ & $0.72\times 10^{-6}$ \\
\bottomrule
\end{tabular}

%% file: outputs/tanh_fixed_point/tanh_results.tex
\begin{tabular}{lrrr}
\toprule
Method & N. evals & $\|T(x) - x\|_\infty$ & $\|x - x_\star\|_\infty$ \\
\midrule
TPA & \textbf{36} & \textbf{$0.53\times 10^{-8}$} & \textbf{$0.18\times 10^{-7}$} \\
Picard & 128 & $0.97\times 10^{-8}$ & $0.91\times 10^{-7}$ \\
SOR ($\omega=1.50$) & 84 & $0.85\times 10^{-8}$ & $0.79\times 10^{-7}$ \\
Anderson(m=2) & 70 & $0.86\times 10^{-8}$ & $0.64\times 10^{-7}$ \\
Anderson(m=3) & 49 & $0.89\times 10^{-8}$ & $0.72\times 10^{-7}$ \\
Anderson(m=5) & 38 & $0.69\times 10^{-8}$ & $0.49\times 10^{-7}$ \\
\bottomrule
\end{tabular}

%% file: outputs/poisson_2d/poisson_results_50.tex
\begin{tabular}{lrrr}
\toprule
Method & N. evals & $\|T(x) - x\|_\infty$ & $\|x - x_\star\|_\infty$ \\
\midrule
TPA & \textbf{244} & $0.97\times 10^{-8}$ & $0.51\times 10^{-5}$ \\
Picard & 4317 & $1.00\times 10^{-8}$ & $0.53\times 10^{-5}$ \\
Anderson(m=2) & 3722 & $1.00\times 10^{-8}$ & $0.51\times 10^{-5}$ \\
Anderson(m=3) & 1573 & $0.99\times 10^{-8}$ & $0.47\times 10^{-5}$ \\
Anderson(m=5) & 955 & \textbf{$0.96\times 10^{-8}$} & \textbf{$0.24\times 10^{-5}$} \\
\bottomrule
\end{tabular}

%% file: paper.bbl
\begin{thebibliography}{1}

\bibitem{aitken_xxv.bernoullis_1927}
A.~C. Aitken.
\newblock {XXV}.—{On} {Bernoulli}'s {Numerical} {Solution} of {Algebraic}
  {Equations}.
\newblock {\em Proc. R. Soc. Edinburgh Sect. A}, 46:289--305, 1927.

\bibitem{anderson_iterative_1965}
Donald~G. Anderson.
\newblock Iterative {Procedures} for {Nonlinear} {Integral} {Equations}.
\newblock {\em J. ACM}, 12(4):547--560, October 1965.

\bibitem{saad_acceleration_2025}
Yousef Saad.
\newblock Acceleration methods for fixed-point iterations.
\newblock {\em Acta Numer.}, 34:805--890, July 2025.

\bibitem{sidi_vector_2008}
Avram Sidi.
\newblock Vector extrapolation methods with applications to solution of large
  systems of equations and to {PageRank} computations.
\newblock {\em Comput. Math. Appl.}, 56(1):1--24, July 2008.

\bibitem{young_iterative_1954}
David Young.
\newblock Iterative methods for solving partial difference equations of
  elliptic type.
\newblock {\em Trans. Amer. Math. Soc.}, 76(1):92--111, 1954.

\end{thebibliography}
